\documentclass[a4paper,reqno,11pt]{amsart}
\usepackage{t1enc}
\usepackage[margin=1in]{geometry}
\usepackage{ifthen}
\usepackage{graphicx}

\newcommand{\Q}{\mathbf{Q}}
\newcommand{\R}{\mathbf{R}}

\newcommand{\pr}{\textbf{P}}
\newcommand{\ex}{\mathbf{E}}

\newcommand{\ind}{\mathbf{1}}

\usepackage{color}

\theoremstyle{plain}
\newtheorem{theorem}{Theorem}

\newtheorem{corollary}{Corollary}
\newtheorem{proposition}{Proposition}

\theoremstyle{definition}
\newtheorem{definition}{Definition}

\newtheorem{remark}{Remark}

\theoremstyle{remark}

\newcommand{\formula}[2][nolabel]
{\ifthenelse{\equal{#1}{nolabel}}
 {\begin{align*} #2 \end{align*}}
 {\ifthenelse{\equal{#1}{}}
  {\begin{align} #2 \end{align}}
  {\begin{align} \label{#1} #2 \end{align}}
 }
}

\numberwithin{equation}{section}

\begin{document}

%
%

\title[Strong  solutions of non-colliding particle systems]{Strong  solutions of non-colliding particle systems}
\thanks{Jacek Ma{\l}ecki was supported by NCN grant no. 2013/11/D/ST1/02622}
\subjclass[2010]{{60J60, 60H15}}
\keywords{{stochastic differential equation, strong solution, non-colliding particle system.}}
\author{Piotr Graczyk, Jacek Ma{\l}ecki}
\address{Piotr Graczyk \\ LAREMA \\ Universit\'e d'Angers \\ 2 Bd Lavoisier \\ 49045 Angers cedex 1, France}
\email{piotr.graczyk@univ-angers.fr}
\address{  Jacek Ma{\l}ecki  \\ Institute of Mathematics and Computer Science \\ Wroc{\l}aw University of Technology \\ ul. Wybrze{\.z}e Wyspia{\'n}\-skiego 27 \\ 50-370 Wroc{\l}aw, Poland}
\email{jacek.malecki@pwr.edu.pl }

\begin{abstract}
We study  systems of stochastic differential equations describing positions $x_1,x_2,\ldots,x_p$
 of $p$ ordered particles, with inter-particles repulsions of the form
$\displaystyle{\frac{H_{ij}(x_i,x_j)}{x_i-x_j}}$. We show the existence of strong and pathwise unique non-colliding solutions of the system with a  colliding initial point
$x_1(0)\leq \ldots\leq x_p(0)$ in the whole generality, under natural assumptions on the coefficients of the equations.
\end{abstract}

\maketitle
\section{Introduction}
Consider the following system of SDEs
\formula[eq:eigenvalues:SDE:general]{
&dx_i = \sigma_i(x_i)dB_i+\left(b_i(x_i)+\sum_{j\neq i}\frac{H_{ij}(x_i,x_j)}{x_i-x_j}\right)dt\/,\quad i=1,\ldots,p\/,\\
& x_1(t)\leq \ldots\leq x_p(t),\ \ \  t\geq 0 \nonumber, 
}
describing positions of $p$  ordered  particles evolving in $\R$. Here  $(B_i)_{i=1,\ldots,p}$ denotes a collection of one-dimensional independent Brownian motions. Throughout the whole paper we assume that the coefficients of the equations are continuous 
and that the functions $H_{ij}$ are non-negative and symmetric in  the  sense (\ref{eq:H:symmetry}).

The SDEs systems   (\ref{eq:eigenvalues:SDE:general})  contain  the following {ones} 
\formula[eq:eigenvalues:SDE]{
  dx_i &= 2g(x_i)h(x_i)dB_i+\beta \left({b}(x_i)+\sum_{j\neq i}\frac{G(x_i,x_j)}{x_i-x_j}\right)dt\/,\quad i=1,\ldots,p\/,
}
where $G(x,y) = g^2(x)h^2(y)+g^2(y)h^2(x)$, $\beta>0$ and $g,h,{b}:\R\to \R$.  Let $S_p$ denote the space of symmetric $p\times p$ real matrices and   $H_p$  the space of Hermitian $p\times p$  
matrices. 
It was shown in \cite{bib:gm13} that for the starting point having no collisions  and for  $\beta=1$, this system describes the eigenvalue processes of the $S_p$-valued process $X_t$ satisfying the following matrix valued stochastic differential equation
\formula{
  dX_t = g(X_t)dW_th(X_t)+h(X_t)dW_t^Tg(X_t)+b(X_t)dt\/,
}
where the functions $g,h,b$ act spectrally on $S_p$ and $W_t$ is a Brownian matrix of dimension $p\times p$.
  When $\beta=2$, the system (\ref{eq:eigenvalues:SDE}) is satisfied by the eigenvalues of the  $H_p$-valued process $Y_t$  which is a solution of
 \formula{
  dY_t = g(Y_t)d\tilde W_t h(X_t)+h(X_t)d\tilde W_t^*g(X_t)+\frac12 b(X_t)dt\/,
}
where $\tilde W_t$ is a complex Brownian matrix of dimension $p\times p$. In the last case, for some special  choices of $g,h$ and $b$, the systems  (\ref{eq:eigenvalues:SDE})
contain the canonical Dyson Brownian Motion $(g=\frac12, h=1, b=0)$ and the eigenvalue processes of the complex Wishart (Laguerre) processes ($g=\sqrt{x}, h=1, b={\rm const}>p-1$).
Recall that  the  Dyson Brownian Motion  is obtained   as $p$ independent Brownian particles conditioned not to collide (see \cite{bib:dyson, grabiner})
and   the Laguerre eigenvalue process  as   $p$ independent  Squared Bessel particles conditioned not to collide (see \cite{bib:konig}).

The general case $\beta\in \R^+$ in (\ref{eq:eigenvalues:SDE}) corresponds to the $\beta$-versions of the processes described by (\ref{eq:eigenvalues:SDE}) with $\beta=1$ and is important in modern statistical physics (see for example \cite{bib:forr}).
On the other hand,  Dyson Brownian Motions are a special case of Brownian particle systems with an interacting potential  (see \cite{bib:RShi}).

Thus the systems    (\ref{eq:eigenvalues:SDE:general}) contain Dyson Brownian Motions, Squared Bessel particle systems, Jacobi particle systems, their $\beta$-versions, non-colliding Brownian and Squared Bessel particles, potential-interacting Brownian particles and other  particle systems crucial in  mathematical physics and physical statistics (see \cite{bib:KatoriSugaku, bib:katori2011}).  Note that the singularities   $\displaystyle{(x_i-x_j)^{-1}}$ make the SDEs system (\ref{eq:eigenvalues:SDE:general})
difficult to solve, especially when the starting point has a collision, i.e. $ x_i(0)= x_j(0)$ for some $i\neq j$. Moreover, the most degenerate case $x_1(0)= \ldots= x_p(0)$ is of great importance in physical applications. 

In this paper we prove the existence of strong and pathwise unique non-colliding solutions of  (\ref{eq:eigenvalues:SDE:general}), with a degenerate colliding initial point 
$x(0)$, in the whole generality, under
natural assumptions on the coefficients of the equations  (\ref{eq:eigenvalues:SDE:general}), formulated and discussed in details in Section \ref{AMR}. The Theorem 1 ensures, that even if starting from the  most degenerate collision state 
$$x_1(0)= \ldots= x_p(0)=0,$$
the particles  $x_i$ will diffract instantly and never more collide. This statement is proved in the strong, trajectorial solution sense.  In this way,
we answer in a very general setting a question raised by Rogers and Shi (\cite[(5i)]{bib:RShi}) in the context of potential-interacting Brownian particles: is the finite particle process well defined by  its SDEs  system?  
As observed by Grabiner in \cite{grabiner}, starting the process from a collision point makes impossible the usual conditioning procedure and the existence of strong solutions  of corresponding SDEs is highly unclear even in the case of processes conditioned not to collide. 

In some particular cases (Dyson Brownian Motions, some  Squared Bessel particle systems) these difficulties have been overcome and the existence of strong solutions of (\ref{eq:eigenvalues:SDE:general}) has been established by C\'epa and L\'epingle in \cite{bib:cepa, bib:cepaEsaim, bib:lep}, using the technique of Multivalued SDEs (MSDEs). 
The   MSDEs theory  was used in
\cite{bib:demni} and \cite{bib:Schapira} in order to show the strong existence of solutions of radial Dunkl and Heckman-Opdam SDEs with more general singularities.

However, the approach via  MSDEs can not be applied to the equations of the general form given in (\ref{eq:eigenvalues:SDE:general}) and consequently, the existence of the strong solutions has been an open question in many important examples. 

Note that some  existence  results were  proved in last years without use of MSDEs (see \cite{bib:Inukai,bib:Angers2006,  bib:Chyb}) but under the hypothesis of starting without collision
(Chybiryakov's proof of   \cite[Prop. 6.8,p.170]{bib:Angers2006}  does not work for a  collision starting point).

Our approach is based on the classical It\^o calculus, applied to elementary symmetric polynomials in $p$ variables  $X=(x_1,\ldots,x_p)$ 
 \formula{
  y_n= e_n(X) = \sum_{i_1<\ldots<i_n}x_{i_1}x_{i_2}\ldots x_{i_n}\/,
 }
as well as to symmetric polynomials of squares of differences between particles
\formula{
  V_n = e_n(A)\/,\quad \textrm{where } A = \{a_{ij}=( x_i- x_j)^2:1\leq i<j\leq p\}\/.
}
 The main advantage of the semimartingales $y_n$ and $V_n$  is that singularities disappear in their stochastic descriptions. Moreover, the processes $V_n$ control the collisions between particles.

In the next Section \ref{AMR} we present and explain technical  assumptions of the main result of the paper, Theorem 1,  formulated at
the end of Section \ref{AMR}. In Sections  \ref{SymPol} and \ref{SymPolDif} we develop the stochastic analysis of symmetric polynomial  stochastic processes $y_n$
and  $V_n$. These sections  provide the main ingredients of the proof of Theorem \ref{thm:main}. In Section \ref{WeakExist} we show that the system (\ref{eq:eigenvalues:SDE:general}) 
has a weak continuous solution.  Next, we prove the pathwise uniqueness of solutions of the system (\ref{eq:eigenvalues:SDE:general}) 
and we conclude with a proof of Theorem \ref{thm:main}.
The last Section \ref{sec:EandA} contains applications to important classes of particle systems.

\section{Assumptions and Main Result}\label{AMR}
As it was mentioned in the Introduction, our general assumptions on the coefficients of the equations are
\begin{itemize}
   \item {\it the functions $\sigma_i, b_i, H_{ij}$ are continuous for every $i,j=1,\ldots,p$ and $i\neq j$;}
   \item {\it the functions $H_{ij}$ are non-negative and the following symmetry condition holds}
	\formula[eq:H:symmetry]{
	   H_{ij}(x,y)=H_{ji}(y,x) \/,\quad x,y\in \R\/.
	}
\end{itemize}
Note that $H_{ij}(x_i,x_j)/(x_i-x_j)$ describes the repulsive force with which the $j$-th particle located in  $x_j$ acts on the $i$-th particle located in  $x_i$. The symmetry assumptions on $H_{ij}$ mean that if $j>i$, i.e. $x_j>x_i$, then  the upper particle  $x_j$ pushes the lower particle $x_i$ down with the same force as the lower one pushes the upper one up.

Next three conditions are   adaptations of standard regularity assumptions occurring in the theory of SDEs (without singularities), which usually guarantee uniqueness and non-explosion of solutions. Note that we remain in the context of one-dimensional Yamada-Watanabe theorem, where the coefficients in the martingale part are allowed to be $1/2$-H\"older continuous and the drift part coefficients are Lipschitz continuous.

\begin{itemize}
\item[(C1)]{\it There exists a function $\rho:\R^+\to\R^+$ such that $\int_{0^+}\rho^{-1}(x)dx=\infty$
and that  
\formula{
  |\sigma_i(x)-\sigma_i(y)|^{2}\leq \rho(|x-y|)\/,\quad x,y\in \R\/, \ i=1,\ldots,p. }
Moreover, the  functions $b_i$ are Lipschitz continuous or non-increasing}
\item[(C2)] {\it There exists $c> 0$ such that}
		\formula{
		   \sigma_i^2(x)+b_i(x)x\leq c(1+|x|^2)\/,\quad x\in \R\/,\\
			 H_{ij}(x,y)\leq c(1+|xy|)\/,\quad x,y\in \R\/.
		}
		\end{itemize}
		The conditions (C2) are standard conditions on the growth of the coefficients of SDE which give finiteness of the solutions for every $t> 0$, however the sublinear growth of $b_i$ can be replaced by non-positivity of $b_i(x)x$ for large $x$.

 The last group of   conditions (A1)-(A5) relates mainly to the singular part of the equations.
Condition (A1) will be crucial for the proof of the pathwise uniqueness of solutions in Section
\ref{WeakExist}. Conditions  (A2)-(A5) are introduced in order  to ensure non-collisions of the particles, which is a crucial property of a solution of (\ref{eq:eigenvalues:SDE:general}) to show the existence of its strong solution.
\begin{itemize}
   \item[(A1)] {\it For every $i\neq j$ and $w<x<y<z$ we have}
	\formula[eq:A1]{
	   H_{ij}(w,z)(y-x)\leq H_{ij}(x,y)(z-w)\/.
	}
	\end{itemize}
	Since $H_{ij}(x_i,x_j)/(x_i-x_j)$ describes the force with which the particles  $x_j$ and  $x_i$ repel each other, the condition (\ref{eq:A1}) means that  the force decreases as the particles move away from each other in such a way that the first particle goes down and the other goes up.
	\begin{itemize}
	 \item[(A2)] {\it There exists $c\geq 0$ such that for every $i\neq j$ we have}
	   \formula[eq:A2]{
	       \sigma_i^2(x)+\sigma_j^2(y)\leq c(x-y)^2+4H_{ij}(x,y)\/,\quad x,y\in \R\/,
				}
\end{itemize}
This assumption ensures that the repulsive forces between the particles are sufficiently large relatively to the martingale part to prevent collisions of the particles caused by $\sigma_i(x_i)dB_i$ and $\sigma_j(x_j)dB_j$. Moreover, this assumption is optimal in many important examples such as $\beta$ versions of Dyson's Brownian motion model (see Section \ref{sec:EandA} for more details), i.e. for $\sigma= 1$, $b= 0$ and $H_{ij}(x,y) = \beta/2$ the condition (A2) holds if and only if $\beta\geq 1$, which is a necessary and sufficient condition for the system to have no collisions.
\begin{itemize}
\item[(A3)] {\it There exists $c\geq 0$ such that for every $x<y<z$ and $i<j<k$ }
	   \formula{
	   H_{ij}(x,y)(y-x) +   H_{jk}(y,z)(z-y)  \leq c(z-y)(z-x)(y-x) + H_{ik}(x,z)(z-x)\/.
				}	
\end{itemize}
This condition is used to ensure that the repulsive forces between particles do not cause collisions. Since $H$ is non-negative, two particles are pushed off from each other, but if we add another particle, which is above the previous ones, then two additional forces appear which push the first two particles down. The condition (A3) implies that the additional forces do not cause a collision between two original particles.   

It can be easily seen that the conditions (A2) and (A3) do not ensure that the particles become immediately distinct if we start from a collision point (consider the example of generalized squared Bessel particles with integer order $\alpha\in \{0,1,\ldots,p-2\}$ starting from zero). Thus, if $x_i(0)=x_j(0)=x$ for some $i<j$, we will distinguish two situations. When $\sigma_i^2(x)+\sigma_j^2(x)>0$ or $H_{ij}(x,x)>0$, then the process $X$ will instantly leave the initial  collision point thanks to the martingale part or the repulsive forces, respectively (see Proposition  \ref{prop:collisiontime:glued}). We call such a phenomenon a "diffraction" of particles. Consequently,  in the study of the particle process we must pay special attention to  starting from a collision in an element of the sets  
\formula{
 G_{kl}=\bigcap_{k<i<j<l}\{x: \sigma_i^2(x)+\sigma_j^2(x) +H_{ij}(x,x)=0\},\quad 1\leq k<l\leq p.
}
 We will call elements of the sets $G_{kl}$ "degenerate points", and if 
$$x_k(0)=x_l(0)=x\in G_{kl},$$ we will say
that the particle process is  starting from a "multiple degenerate" point.
The next condition    (\ref{eq:Degeneration})
guarantees that 
there is a force coming from the whole drift at $x$  such that at most one particle can stay at  the point $x$, so the multiple degeneracy will disappear
(Proposition \ref{prop:Degenerated}).
\begin{itemize}
   \item[(A4)]  {\it  The sets $G_{kl}$ consist of isolated points and for every $x\in G_{kl}$ we have }
	\formula[eq:Degeneration]{
	   \sum_{i=k}^l\left( b_i(x)+\sum_{j=1}^{p-2}\frac{H_{ij}(x,y_j)}{x-y_j}\ind_{\R\setminus \{x\}}(y_j)\right)\neq 0,
		}
	{\it for every $y_1,\ldots,y_{p-2}\in \R $.}
\end{itemize}	
 We use the convention that multiplying by the indicator $\ind_{\R\setminus \left\{x\right\}}(y_j)=0$  always gives $0$, i.e. the whole  $j$-th term of the second sum in 
(\ref{eq:Degeneration})
disappears when $y_j=x$.
 
Finally, we consider the following monotonicity property  of the drift coefficients $b_i$ in (\ref{eq:eigenvalues:SDE:general})
\begin{itemize}
 \item[(A5)] If $i<j$ then $b_i(x)\leq b_j(x)$ for all $x\in \R$.
\end{itemize}
This condition comes up naturally because if $b_i(x)> b_j(x)$ then the particle $x_i$ could catch up with the particle $x_j$ thanks to the bigger drift force.

In the case, when the coefficients of the equations do not depend on $i$ and $j$, i.e. $\sigma_i(x)=\sigma(x)$, $b_i(x)=b(x)$ and $H_{ij}(x,y)=H(x,y)$, simple sufficient conditions for (A1)-(A5)  are discussed in more detail in Section \ref{sec:EandA}. 

\remark If we know that the particle system $(x_1,\ldots,x_p)$ lives on some subset  $I=[a,b]\subset\R$, then we can restrict all the conditions to $x,y\in I$.

Before formulating Theorem \ref{thm:main}, the main result of the paper, recall that, accordingly to \cite[IX(1.2)]{bib:ry99}, a pair $(X,B)$ is a solution of 
the system  (\ref{eq:eigenvalues:SDE:general}) if all the integrals appearing in its integral form are meaningful. In particular, 
 the integrals of the  drift parts of (\ref{eq:eigenvalues:SDE:general}) will be understood, if needed, as improper Riemann integrals.

\begin{theorem}
\label{thm:main}
Consider the system  (\ref{eq:eigenvalues:SDE:general})  with  an initial condition
\formula{
x_1(0)\leq \ldots\leq x_p(0).
}
If the conditions (C1), (C2) and (A1)-(A5) hold, then there exists a unique strong non-exploding solution $[X(t)]_{t\geq 0}$ of (\ref{eq:eigenvalues:SDE:general}) such that the first collision time
\formula{
T = \inf\{t>0: x_i(t)=x_j(t)\textrm{ for some $i\neq j$, $i,j=1,\ldots,p$}\}
}
is infinite almost surely.
\end{theorem}

\remark If we drop the condition (C2) in Theorem \ref{thm:main}, then there exists a unique strong solution, possibly admitting explosions, such that the first collision time $T$ is not shorter than the lifetime of the solution. 

\remark It is enough to assume (A2)-(A5) to show that there exists a solution of (\ref{eq:eigenvalues:SDE:general}) having no collisions after the start. The additional conditions (C1) and (A1) ensure the pathwise uniqueness of the solutions and consequently the existence of a unique strong solution.


\section{Stochastic description of the basic symmetric polynomials}\label{SymPol}
We denote the elementary symmetric polynomials in $p$ variables  $X=(x_1,\ldots,x_p)$ and of  degree $n=1,2,\ldots,p$ by
 \formula{
   e_n(X) = \sum_{i_1<\ldots<i_n}x_{i_1}x_{i_2}\ldots x_{i_n}\/.
 }
We use the following notational conventions: $e_0(X)= 1$ and $e_{-1}(X)=0$. We will also consider incomplete polynomials. For any fixed collection $x_{j_1}, x_{j_2},\ldots,x_{j_k}$ of entries of $X$
 \formula{
   e_n^{\overline{x}_{j_1}, \overline{x}_{j_2},\ldots,\overline{x}_{j_k}}(X) = \sum_{\stackrel{i_1<\ldots<i_n}{i_k\neq j_l}}x_{i_1}x_{i_2}\ldots x_{i_n}\/,
 }
i.e. it is the sum of all products of length $n$ which do not contain any of the specified variables $x_{j_1}, x_{j_2},\ldots,x_{j_k}$.
There is obviously no one-to-one correspondence between $(x_1,\ldots,x_p)$ and $(e_1,\ldots e_p)$ since changing the order of the arguments does not affect the values of their symmetric polynomials. But if we  restrict the arguments  to the open set 
\formula{
  C_+ = \{(x_1,\ldots,x_p)\in\R^p: x_1<x_2<\ldots<x_p\}\/
} 
then the smooth function
\formula{
  e = (e_1,\ldots,e_p): C_+\to \R^p
}
is one-to-one. This follows from the fact that $(-1)^k e_k(X)$ is the coefficient of $x^{p-k}$ in the polynomial $P(x)=\prod_{i=1}^p (x-x_i)$.
Thus $e$ is a diffeomorphism between $ C_+$ and $ e(C_+)$, which is an open subset of $\R^p$.
Let us denote by
 $$f=(f_1,\ldots,f_p): e(C_+)\to C_+$$
 the inverse diffeomorphism. By the continuity of ordered roots of a polynomial as functions of its coefficients
(see for example \cite{bib:Lojasiewicz}), $f$  extends to  a continuous function
\formula{
   f:\overline{e(C_+)} \stackrel{1-1}{\longrightarrow} \overline{C_+}.
}

\subsection{Symmetric polynomials of particles}
In the following proposition we determine the SDEs system for the symmetric polynomials in $(x_1,\ldots,x_p)$  verifying the system (\ref{eq:eigenvalues:SDE:general}) whenever there are no collisions between particles. 

In the proof, as well as in some other proofs in this paper, we use the property
\formula{
\sum_{i=1}^p a_i \sum_{j\not=i} b_{ij}=\sum_{i<j}(a_i  b_{ij}+a_j  b_{ji}).
} 
\begin{proposition}\label{polynomials}
Let $X=(x_1,\ldots,x_p)$ be a solution of (\ref{eq:eigenvalues:SDE:general}) such that $x_1(0)<\ldots<x_p(0)$. Then the symmetric polynomials 
$e_n(X)$, $n=1,\ldots, p$, are continuous semimartingales described until the first collision time of $(x_i)_{i=1,\ldots,p}$  by the system of SDEs
\begin{eqnarray}
  de_n(X) &=& \left(\sum_{i=1}^p\sigma_i^2(x_i)(e_{n-1}^{\overline{x}_i}(X))^2\right)^{1/2}dU_n \nonumber\\
	&&+\left(\sum_{i=1}^pb_i(x_i)e_{n-1}^{\overline{x}_i}(X)-\sum_{i<j}H_{ij}(x_i,x_j)e_{n-2}^{\overline{x}_i,\overline{x}_j}(X)\right)dt\/,\quad
	n=1,\ldots,p, \label{eq:en:SDE1}
 \end{eqnarray}
 where $\{U_n;n=1\ldots,p\}$ is a family of one-dimensional Brownian motions such that 
 \formula[eq:en:SDE2]{
   d\left<e_n(X),e_m(X)\right> = \sum_{i=1}^p \sigma_i^2(x_i)e_{n-1}^{\overline{x}_i}(X)e_{m-1}^{\overline{x}_i}(X)dt\/.
 }
 \end{proposition}
 \begin{proof}
 By an application of It\^o formula and the fact that $e_n(X)=x_i e_{n-1}^{\overline{x}_i}(X)+e_n^{\overline{x}_i}(X)$
for any $i$ we get that for every $t$ smaller than the first collision time of $(x_i)_{i=1,\ldots,p}$
 \formula{
   de_n(X) &= \sum_{i=1}^p e_{n-1}^{\overline{x}_i}(X)dx_i\\
   &=\sum_{i=1}^p \sigma_i(x_i)e_{n-1}^{\overline{x}_i}(X)dB_i+\left(\sum_{i=1}^p b_i(x_i)e_{n-1}^{\overline{x}_i}(X)+\sum_{i=1}^p\sum_{j\neq i} e_{n-1}^{\overline{x}_i}(X)\frac{H_{ij}(x_i,x_j)}{x_i-x_j}\right)dt.
 }
 Thus there exist Brownian motions $U_n$, $n=1,\ldots,p$, such that
 \formula{
    \sum_{i=1}^p \sigma_i(x_i)e_{n-1}^{\overline{x}_i}(X)dB_i = \left(\sum_{i=1}^p\sigma_i^2({x}_i)(e_{n-1}^{\overline{x}_i}(X))^2\right)^{1/2}dU_n
 }
 and (\ref{eq:en:SDE2}) holds. Moreover, by the  symmetry property (C2) of $H_{ij}(x,y)$ and the fact that for any $j\not=i$ 
 \formula{
 e_{n-1}^{\overline{x}_i}(X) = x_j e_{n-2}^{\overline{x}_i,\overline{x}_j}(X)+e_{n-1}^{\overline{x}_i,\overline{x}_j}(X)\/,
 }
  we obtain
 \formula{
    \sum_{i=1}^p\sum_{j\neq i} e_{n-1}^{\overline{x}_i}(X)\frac{H_{ij}(x_i,x_j)}{x_i-x_j} &= \sum_{i<j}\left(e_{n-1}^{\overline{x}_i}(X)\frac{H_{ij}(x_i,x_j)}{x_i-x_j}+e_{n-1}^{\overline{x}_j}(X)\frac{H_{ji}(x_j,x_i)}{x_j-x_i}\right)\\
     &= -\sum_{i<j}e_{n-2}^{\overline{x}_i,\overline{x}_j}H_{ij}(x_i,x_j)\/.
 }
 This ends the proof.
 \end{proof}
Note the following remarkable property of the stochastic differential equations describing the polynomial processes $e_n(X)$, $n=1,\ldots,p$: the singularities $(x_i-x_j)^{-1}$ appearing in (\ref{eq:eigenvalues:SDE:general}) are no longer present in (\ref{eq:en:SDE1}).

Now, using the map  $f:\overline{e(C_+)} \rightarrow \overline{C_+}$  we get rid of $x_i$'s in the system (\ref{eq:en:SDE1}). We will shorten the notation $e_n^{\overline{x}_i}$ to
 $e_n^{\overline{i}}$. We denote by $y$ elements of $\overline{e(C_+)}$.

\begin{proposition}
Define the following functions on $\overline{e(C_+)}$:
\formula{
  a_n(y) &= \left(\sum_{i=1}^p\sigma_i^2(f_i(y))(e_{n-1}^{\overline{i}}(f(y)))^2\right)^{1/2}\/,\ \ y\in \overline{C_+},\\
	q_n(y) &= \sum_{i=1}^p b_i(f_i(y))e_{n-1}^{\overline{i}}(f(y))-\sum_{i<j}e_{n-2}^{\overline{i},\overline{j}}(f(y))H_{ij}(f_i(y),f_j(y))\/,\\
	s_{n,m}(y) &= \sum_{i=1}^p \sigma_i^2(f_i(y))e_{n-1}^{\overline{i}}(f(y))e_{m-1}^{\overline{j}}(f(y)).
}

(i) The functions $a_n, q_n$ and $s_{n,m}$, $n,m=1,\ldots,p$, are continuous on $\overline{e(C_+)}$.

(ii) Let $X=(x_1,\ldots,x_p)$ be a solution of (\ref{eq:eigenvalues:SDE:general}) on ${C_+}$. The  symmetric polynomial processes $y_n(t)=e_n(X(t))$, $t\geq 0, n=1,\ldots,p$,  satisfy the system of SDEs
\begin{eqnarray}
   dy_n = a_n(y_1,\ldots,y_p)dU_n +q_n(y_1,\ldots,y_p)dt\/,\quad n=1,\ldots,p\/, \label{eq:en:SDEs}
	\end{eqnarray}
where $\{U_n;n=1\ldots,p\}$ is a family of one-dimensional Brownian motions satisfying
\formula[bracket]{	
	 \left<a_n dU_n,a_m dU_m\right> = s_{n,m}dt\/,\quad n\neq m\/,\quad n,m=1,\ldots,p\/.
}
(iii)  Let $y_0\in \overline{e(C_+)}$. The   system (\ref{eq:en:SDEs})-(\ref{bracket}) with the initial condition
$y(0)=y_0$ has a solution, possibly admitting explosions.
\end{proposition}
\begin{proof}
The part (i) follows by 
 continuity of the map  $f$  on $\overline{e(C_+)}$
and by continuity of the functions
 $\sigma_i$, $b_i$ and $H_{ij}$. Part (ii) is a corollary of Proposition \ref{polynomials}. By \cite[Th.2.3, p.159]{bib:IW},  the  system (\ref{eq:en:SDEs})
with the condition (\ref{bracket})
has a solution, possibly admitting explosions, for every $y(0)\in \overline{e(C_+)}$. 
\end{proof}

\begin{definition}\label{fY}
Let  $(y_1,\ldots,y_p)$ be a solution of (\ref{eq:en:SDEs})-(\ref{bracket}) with $y(0)\in \overline{e(C_+)}$.
We define stochastic process $\Lambda=(\lambda_1,\ldots,\lambda_p)$, where
\formula{
\lambda_i=f_i(y_1,\ldots,y_p),\ \  i=1,\ldots, p.
}
\end{definition}
Thus, from now on, whenever we write $\lambda_i(t)$ we mean the process $\lambda_i(t)=f_i(y(t))$, 
defined from a solution $y(t)$  of (\ref{eq:en:SDEs})-(\ref{bracket}), using the inverse symmetric polynomial map $f$. Obviously we have
$$
y_n=e_n(\lambda_1,\ldots,\lambda_p)=e_n(\Lambda),\ \  n=1,\ldots, p
$$
and whenever $y(0)\in e(C_+)$, i.e. $\lambda_i(0)\neq \lambda_j(0)$ for every $i\neq j$, then $\Lambda=(\lambda_1,\ldots,\lambda_p)$ is a solution of (\ref{eq:eigenvalues:SDE:general}) up to the first collision time. It is thus natural to interpret  $\Lambda=(\lambda_1,\ldots,\lambda_p)$  as a system of particles related to the solution $y$ of  (\ref{eq:en:SDEs})-(\ref{bracket}).
\subsection{Non-explosion of solutions}
Now we show that the condition (C2) is sufficient in order that the solutions $y_1,\ldots,y_p$ of  (\ref{eq:en:SDEs})-(\ref{bracket}) do not explode in a finite time. 
\begin{proposition}
  If (C2) holds, then the explosion time of any solution of (\ref{eq:en:SDEs}) is infinite almost surely.
\end{proposition}
\begin{proof}
   Let $y=(y_1,\ldots,y_p)$ be a solution of (\ref{eq:en:SDEs}). We define 
	\formula{
	   R_t=\sum_{i=1}^p \lambda_i^2  =y_1^2-2y_2. 
	 }
 Applying It\^o formula to (\ref{eq:en:SDEs}) we 
get
	\formula{
	   dR_{t} = 2y_1a_1dU_1-2a_2 dU_2
		+\left(\sum_{i=1}^p(\sigma_i^2(\lambda_i)+2\lambda_ib_i(\lambda_i))+2\sum_{i<j}H_{ij}(\lambda_i,\lambda_j)\right)dt\/.
	}
	Using (\ref{bracket}) one sees easily that $\left<y_1a_1dU_1-a_2 dU_2, y_1a_1dU_1-a_2 dU_2\right>= 
	\sum_{i=1}^p\sigma_i^2(\lambda_i)\lambda_i^2$.  
	It follows that there exists a Brownian motion $W_t$ such that 
		$$dR_{t} =2 
	\left(\sum_{i=1}^p\sigma_i^2(\lambda_i)\lambda_i^2\right)^{1/2}dW_t
	+\left(\sum_{i=1}^p(\sigma_i^2(\lambda_i)+2\lambda_ib_i(\lambda_i))+2\sum_{i<j}H_{ij}(\lambda_i,\lambda_j)\right)dt\/.
	$$
	The rest of the proof is similar to the proof of the classical theorem on non-explosion of solutions of a SDE, see
	Theorem 2.4 in \cite{bib:IW}. For the convenience of the reader we provide the proof.
	
	Set $\tau_n = \inf\{t>0: R_t\geq n\}$.
	Using  the fact that the expectation of the martingale part vanishes and (C2)  we get
	\formula{
	  \ex R_{t\wedge \tau_n} &= R_0+\ex \int_0^{t\wedge \tau_n}\left(\sum_{i=1}^p(\sigma_i^2(\lambda_i)+2\lambda_ib_i(\lambda_i))+2\sum_{i<j}H_{ij}(\lambda_i,\lambda_j)\right)ds\\
		&\leq R_0+ c\int_0^{t} (1+\ex R_{s\wedge \tau_n})ds
	}
	Continuity of the paths and the Lebesgue dominated convergence theorem imply that the function $t\to \ex R_{t\wedge \tau_n}$ is continuous. By the integral version of the Gronwall's lemma 
	\formula{
	  \ex R_{t\wedge \tau_n} \leq (1+R_0)e^{ct}-1\/,\quad t\geq 0
	}
 and taking $n\to \infty$ we obtain that $R_t$ is finite almost surely for every $t\geq 0$. Thus, all the processes 
$\lambda_i^2(t)\leq R_t$ are finite for every $t$ and consequently every $y_n(t)=e_n(\Lambda(t))$ is finite  almost surely. It implies that the explosion time of the solution  $y(t)$ is infinite almost surely.
\end{proof}
\subsection{ Instant exit from  a multiple degenerate point}
Starting from this section, we study  the behavior of the particles $\lambda_i$ associated to a solution $y$ of (\ref{eq:en:SDEs}).

In the next proposition we use the condition (A4) to show that if there are at least two particles $\lambda_k, \lambda_l$ starting from the same degenerate point, i.e. a point $x$ belonging to a set $G_{kl}$,
 then immediately all  particles, except perhaps one, are pushed off that point, i.e. there might be at most one particle which stays at $x$. Observe that Proposition \ref{prop:Degenerated} does not imply the instant diffraction of the particles, i.e. it could possibly happen that $\lambda_k(t)=\lambda_l(t)$ on some time interval $(0,\epsilon)$ with positive probability. This problem together with the non-degenerate case, i.e. starting from a collision in $x\notin G_{kl}$, will be considered later in Proposition \ref{prop:collisiontime:glued}.

\begin{proposition}
\label{prop:Degenerated}
   Let $y=(y_1,\ldots,y_p)$ be a solution of (\ref{eq:en:SDEs}) and assume that (A4) holds.
	If $\lambda_k(0)=\lambda_l(0)=x\in G_{kl}$ for some $k< l$  then 
	\formula{
	   \tau=\inf\{t>0:\sum_{i=1}^p \ind_{\{\lambda_i(t)=x\}}\leq 1\} = 0 \quad\textrm{a.s.}
	}
\end{proposition}
\begin{proof}
Without loss of generality we can assume that $\lambda_{k-1}(0)<\lambda_k(0)$ and $\lambda_l(0)<\lambda_{l+1} (0)$, i.e. there are exactly $l-k+1$ particles starting from $x$, and we denote by $S=\{k,k+1,\ldots,l-1,l\}$ the set of indices of these particles. Moreover, we can and do assume that $x=0$. Additionally, we will denote by $\gamma(t)=\sum_{i=1}^p \ind_{\{\lambda_i(t)=x\}}$ the number of particles staying at time $t$ in $x$. Let $\gamma = \gamma(0) = l-k+1$ and note that we have $\gamma\geq 2$. Suppose by contradiction  that the first time, when at least one of the particles moves from $0$, i.e. the random variable
	\formula{
	  \tau_0= \inf\{t>0: \lambda_k(t)\neq 0 \vee \ldots \vee \lambda_{k+\gamma-1}(t)\neq 0\}
	}
	is greater than zero with positive probability. Using continuity of the paths we can see that there exists $\tilde{\tau}_0$ (positive with positive probability) such that 
	\formula{
	\lambda_1(t)\leq \lambda_{k-1}(t)<\lambda_k(t)=\ldots=\lambda_{l}(t)=0<\lambda_{l+1}(t)\leq \ldots\leq \lambda_p(t)\/,\quad t<\tilde{\tau}_0\/,
	}
	i.e. we have $\gamma$ particles remaining equal to zero on $[0,\tilde{\tau}_0)$ and $p-\gamma$ nonzero particles on this time interval. Since each product forming part of the sum defining $e_n(\Lambda)$, where $n=p-\gamma+1$ contains at least one of the zero particle, we have $e_{n}(\Lambda)(t)=0$ for $t\in [0,\tilde{\tau}_0)$. In particular its drift part is zero on $[0,\tilde{\tau}_0)$.
	
	From the other side, observe that if $i\notin S$ then $e_{n-1}^{\overline{\lambda}_i}(\Lambda) = 0$ and $e_{n-2}^{\overline{\lambda}_i,\overline{\lambda}_j}(\Lambda)=0$ if in addition $j\notin S$. Since $H(0,0)=0$ we can write the drift part of $y_n$ on $[0,\tilde{\tau}_0)$ as
	\formula{
	  \textrm{drift}[y_n] =  \sum_{i\in S}\left(b_i(\lambda_i)e_{n-1}^{\overline{\lambda}_i}(\Lambda)-\sum_{j\neq i,\/j\notin S}e_{n-2}^{\overline{\lambda}_i,\overline{\lambda}_j}(\Lambda)H_{ij}(\lambda_i,\lambda_j)\right)\/.
	}
	Seeing that for $j\notin S$ we have $e_{n-1}^{\overline{\lambda}_i,\overline{\lambda}_{j}}(\Lambda) = 0$, we can write 
	\formula{
	  e_{n-1}^{\overline{\lambda}_i}(\Lambda) = \lambda_j e_{n-2}^{\overline{\lambda}_i,\overline{\lambda}_j}+e_{n-1}^{\overline{\lambda}_i,\overline{\lambda}_{j}}(\Lambda) = \lambda_j e_{n-2}^{\overline{\lambda}_i,\overline{\lambda}_j}\/,\textrm{ whenever } i\in S, j\notin S\/. 
	}
	Note also that for every $i\neq k$, $i,k\in S$ we have $e_{n-1}^{\overline{\lambda}_i}(\Lambda) = e_{n-1}^{\overline{\lambda}_k}(\Lambda)$ and this common quantity is equal to the product of $p-\gamma$ non zero particles. Consequently, we obtain that for $t<\tilde{\tau}_0$ the drift part of $y_n=e_n(\Lambda)$ is from one side identically zero, but from the other side it is equal to 
	\formula{
	   \textrm{drift}[y_n] &= \sum_{i\in S} e_{n-1}^{\overline{\lambda}_i}(\Lambda)\left(b_i(0)+\sum_{j\neq i,j\notin S}
		\frac{H_{ij}(0,\lambda_j)}{-\lambda_j}\right)\\
		&= e_{n-1}^{\overline{\lambda}_k}(\Lambda)\sum_{i\in S}\left(b_i(0)+\sum_{j\neq i}\frac{H_{ij}(0,\lambda_j)}{-\lambda_j}\ind_{\{\lambda_j\neq 0\}}\right)\/,\quad k\in S\/.
	}
Since the assumption (A4) ensures that  the above-given sum over $S$ is nonzero, it implies that for every $i\in S$ we have $e_{n-1}^{\overline{\lambda}_i}(\Lambda) =0$ on $(0,\tilde{\tau}_0)$, which is a contradiction with the fact that $e_{n-1}^{\overline{\lambda}_i}(\Lambda)$ is equal to the product of particles which are nonzero on $[0,\tilde{\tau}_0)$.

Now we reason by induction on $\gamma=\gamma(0)\geq 2$ in order to prove the statement of the Proposition.  Note that if $\gamma=2$, we have $\tau=\tau_0$.	
The above-given arguments show that in this situation, for every $t>0$  we have $\pr_{2}(\tau>t)=0$.

Suppose $\gamma=\gamma(0)> 2$. Using continuity of the paths and the Markov property we obtain  
\formula{
   \pr_{\gamma}(\tau>t) &\leq \sum_{q\in [0,t)\cap\Q}\pr_{\gamma}(\gamma(q)\leq\gamma-1,\tau>t)\\
	&= \sum_{q\in [0,t)\cap\Q}\ex_{\gamma}[\pr_{\gamma(q)}(\tau>t-q), \gamma(q)\leq\gamma-1]=0\/
}
by induction hypothesis. Consequently, for all $t>0$ one has  $\pr_{\gamma}(\tau>t)=0$ for every $\gamma = \gamma(0)\geq 2$. The Proposition follows.
\end{proof}

\section{Polynomials of squares of differences between particles and collision times}\label{SymPolDif}
\subsection{Symmetric polynomials of squares of differences between particles}
 For any $1\leq i,j\leq p$ we put $a_{ij} = (\lambda_i-\lambda_j)^2$ and define the family of the processes 
\formula{
  V_n = e_n(A)\/,\quad \textrm{where } A = \{a_{ij}=(\lambda_i-\lambda_j)^2:1\leq i<j\leq p\}\/,
}
where $n=1,2,\ldots,N=p(p-1)/2$. The process $V_n$ is a sum of all products of the length $n$ of the squared differences between particles. In particular
\formula{
  V_{N}(t) = \prod_{i<j}(\lambda_i(t)-\lambda_j(t))^2
}
is the squared Vandermonde determinant and 
\formula{
  V_{1}(t) = \sum_{i<j}(\lambda_i(t)-\lambda_j(t))^2\/.
}
Note that these processes and their zeros control the collisions between particles. For example, $V_{N}(t)$ is zero if and only if any collision occurs at time $t$, $V_1(t)$ is zero if and only if all the particles are equal at time $t$. In general, $V_n(t)$ is zero if and only if at least $N-n+1$ collisions take place at time $t$. 

Since $V_n$ as a function of $(\lambda_1,\ldots,\lambda_p)$ is a symmetric polynomial, it can by expressed as a polynomial  of 
$ y_1=e_1(\Lambda),\ldots,y_p=e_p(\Lambda)$.
Processes $y_1,\ldots, y_p$ are  defined as a solution of the system (\ref{eq:en:SDEs})-(\ref{bracket}), so by It\^o formula,  $V_n$ are semimartingales. This is a reason why we consider the squares of differences between particles instead of studying the differences themselves. We begin with exploring the semimartingale structure of $V_n$, $n=1,\ldots,N$ in the case when $V_N(0)>0$, i.e. the system starts from a point having no collisions.  Then the semimartingale $\Lambda=(\lambda_1,\ldots,\lambda_p)$ is a solution of the system (\ref{eq:eigenvalues:SDE:general}) up to the first collision time. Thus we can apply the It\^o formula to (\ref{eq:eigenvalues:SDE:general}) which significantly simplifies calculations.

 Since $e_n(A)=a_{ij}e_{n-1}^{\overline{a}_{ij}}(A)+e_n^{\overline{a}_{ij}}(A)$ and $e_{n-1}^{\overline{a}_{ij}}(A)=a_{kl}e_{n-2}^{\overline{a}_{ij},\overline{a}_{kl}}(A)+e_{n-1}^{\overline{a}_{ij},\overline{a}_{kl}}(A)$, the It\^o formula implies
	\formula{
	 dV_n = \sum_{i<j}e_{n-1}^{\overline{a}_{ij}}(A)da_{ij}+\frac12 \sum_{i<j}\sum_{k<l}e_{n-2}^{\overline{a}_{ij},\overline{a}_{kl}}(A)d\left<a_{ij},a_{kl}\right>\/.
	}
	Note that $d\left<a_{ij},a_{kl}\right>\neq 0$ if and only if  $\{i,j\}\cap\{k,l\}\neq \emptyset$. 
	Moreover
		\formula{
	  da_{ij} &= 2(\lambda_i-\lambda_j)d\lambda_i+2(\lambda_j-\lambda_i)d\lambda_j+d\left<\lambda_i,\lambda_i\right>+d\left<\lambda_j,\lambda_j\right>\\
		&= 2(\lambda_i-\lambda_j)\sigma_i(\lambda_i)dB_i+2(\lambda_j-\lambda_i)\sigma_j(\lambda_j)dB_j+2(\lambda_i-\lambda_j)(b_i(\lambda_i)-b_j(\lambda_j))dt\\
		& +(\sigma_i^2(\lambda_i)+\sigma_j^2(\lambda_j))dt+ 2(\lambda_i-\lambda_j)\sum_{k\neq i}\frac{H_{ik}(\lambda_i,\lambda_k)}{\lambda_i-\lambda_k}dt+2(\lambda_j-\lambda_i)\sum_{k\neq i}\frac{H_{jk}(\lambda_j,\lambda_k)}{\lambda_j-\lambda_k}dt\/,
	}
	which gives
	\formula{
	   d\left<a_{ij},a_{ik}\right> = 4(\lambda_i-\lambda_j)(\lambda_i-\lambda_k)\sigma_i^2(\lambda_i)dt\/.
	}
	Consequently, the martingale part of $dV_n$ is described by
	\formula[eq:Vn:martingale]{	
	  dM_n &= 2\sum_{i=1}^p \sum_{j\neq i}(\lambda_i-\lambda_j)e_{n-1}^{\overline{a}_{ij}}(A)\sigma_i(\lambda_i)dB_i\/,\\
		\nonumber
		d\left<M_n,M_n\right> &= 4\sum_{i=1}^p\sigma_i^2(\lambda_i)\left(\sum_{j\neq i}(\lambda_i-\lambda_j)e_{n-1}^{\overline{a}_{ij}}(A)\right)^2dt\/.
	}
Finally, the drift part of $V_n$ is   (using the notation $i\neq j\neq k$  when no equality between any  two of  the indices $i,j,k$ holds)
	\formula[eq:Vn:drift]{
	  D_ndt &= \sum_{i=1}^p \sigma_i^2(\lambda_i)\sum_{j\neq i}e_{n-1}^{\overline{a}_{ij}}(A)dt+4\sum_{i<j}H_{ij}(\lambda_i,\lambda_j)e_{n-1}^{\overline{a}_{ij}}(A)dt\\
		\nonumber
		&+2\sum_{i=1}^p\sum_{j\neq k\neq i}(\lambda_i-\lambda_j)(\lambda_i-\lambda_k)e_{n-2}^{\overline{a}_{ij},\overline{a}_{ik}}(A)\sigma_i^2(\lambda_i)dt\\
		&+2\sum_{i\neq j\neq k}(\lambda_i-\lambda_j)\frac{H_{ik}(\lambda_i,\lambda_k)}{\lambda_i-\lambda_k}e_{n-1}^{\overline{a}_{ij}}(A)dt+2\sum_{i<j}(\lambda_j-\lambda_i)(b_j(\lambda_j)-b_i(\lambda_i))e_{n-1}^{\overline{a}_{ij}}(A)dt 	\nonumber\/.
	}
Observe that using the relation $e_{n-1}^{\overline{a}_{ij}}(A) = a_{ik}e_{n-2}^{\overline{a}_{ij},\overline{a}_{ik}}(A)+e_{n-1}^{\overline{a}_{ij},\overline{a}_{ik}}(A)$ and the symmetry property (C2), the fourth term can be written as
\formula[eq:Vn:drift:4th_term]{
&2\sum_{i\neq j\neq k}(\lambda_i-\lambda_j)\frac{H_{ik}(\lambda_i,\lambda_k)}{\lambda_i-\lambda_k}e_{n-1}^{\overline{a}_{ij}}(A)dt= \\
   &2\sum_{i\neq j\neq k}(\lambda_i-\lambda_j)(\lambda_i-\lambda_k){H_{ik}(\lambda_i,\lambda_k)}e_{n-2}^{\overline{a}_{ij},\overline{a}_{ik}}(A)dt+2\sum_{i<k}\sum_{j\neq i,j\neq k}H_{ik}(\lambda_i,\lambda_k)e_{n-1}^{\overline{a}_{ij},\overline{a}_{ik}}(A)dt \nonumber\/,
}
so it is well-defined even if any collision occurs. Moreover, the second part is non-negative and the first one vanishes when $V_n=0$. 

Similarly as for SDEs for symmetric polynomials $y_i=e_i(X)$, the stochastic differential equations describing $V_n(\Lambda)$ do not contain singularities $(x_i-x_j)^{-1}$.

 We finish by observing  that the polynomials $V_n$ fulfill equations (\ref{eq:Vn:martingale}) and (\ref{eq:Vn:drift}) also if $V_N(0)=0$, with $\lambda_i=f_i(Y), i=1,\dots,p$. This follows from the fact that the polynomials $V_n$ are smooth polynomial functions of the semimartingales $y_1,\dots,y_p$, satisfying the SDEs system  
(\ref{eq:en:SDEs}). Thus, by the unicity of the martingale and drift part of a semimartingale, the It\^o formula applied
to (\ref{eq:en:SDEs}) gives the equations (\ref{eq:Vn:martingale}) and (\ref{eq:Vn:drift}). But computing derivatives of the functions $V_n=V_n(f(Y))$ required in the It\^o formula
does not depend on the initial condition $V_N(0)$. Such  argument allows us to avoid looking for explicit  relations between the polynomials $V_n$ and $e_n$. A similar argument was used  in \cite{bib:b91}. 

We resume the results of this subsection in the following proposition.
\begin{proposition}
The semimartingales $V_n$, $n=1,\dots, N$, with $V_n(0)\in V_n(\overline{C_+})$ decompose into the martingale and drift part
$$
V_n=M_n+D_n
$$ with  $M_n$ given by  (\ref{eq:Vn:martingale}) and $D_n$ given by (\ref{eq:Vn:drift}) and (\ref{eq:Vn:drift:4th_term}), where $\lambda_i=f_i(Y)$ and
$Y=(y_1,\dots,y_p)$ is a solution of the SDEs system (\ref{eq:en:SDEs}). 
\end{proposition}

\subsection{ Collision time when starting from a regular state}
\label{sec:CollisionTime}
In this section we consider the first collision time defined in terms of the semimartingale $V_N$ by
\formula{
  T = \inf\{t>0: V_N(t)=0\}\/,
	}
with standard convention that $\inf\emptyset = \infty$. 

We begin with a generalization of  Theorem 5 from \cite{bib:gm13}. We show that under certain conditions on coefficients of the equation the particles never collide when the starting point does not have any collisions. 
\begin{proposition}
\label{prop:collisiontime:distinct}
  Let $(V_1,\ldots,V_N)$ be semimartingales described by equations (\ref{eq:Vn:martingale}), (\ref{eq:Vn:drift}) and (\ref{eq:Vn:drift:4th_term}).
	Suppose that $V_N(0)>0$ and (A2) together with (A3) hold. If the functions $x\to b_i(x)$ are Lipschitz continuous or non-decreasing and they satisfy condition (A5), then $T=\infty$ almost surely.  
\end{proposition}

\begin{proof}
The proof is similar to the proof of Theorem 5 from \cite{bib:gm13}.
   Defining $U_t=-\frac12 \ln V_N(t)$ on $[0,T)$ and applying the It\^o formula we obtain the martingale part of $U$  equal to 
	\formula{
	\sum_{i<j}\frac{\sigma_i^2(\lambda_i)dB_i-\sigma_j^2(\lambda_j)dB_j}{\lambda_i-\lambda_j}
	}
	and the following representation of the finite-variation part
	\formula{
	\textrm{drift}[U]_t &= \sum_{i<j}\frac{b_i(\lambda_i)-b_j(\lambda_j)}{\lambda_j-\lambda_i}+\frac12\sum_{i< j} \frac{\sigma_i^2(\lambda_i)+\sigma_j^2(\lambda_j)-4H_{ij}(\lambda_i,\lambda_j)}{(\lambda_j-\lambda_i)^2}\\
		&+\sum_{i<j<k}\frac{H_{jk}(\lambda_j,\lambda_k)(\lambda_k-\lambda_j)-H_{ik}(\lambda_i,\lambda_k)(\lambda_k-\lambda_i)+H_{ij}(\lambda_i,\lambda_j)(\lambda_j-\lambda_i)}{(\lambda_k-\lambda_j)(\lambda_k-\lambda_i)(\lambda_j-\lambda_i)}\/.
	}
	
	Note that conditions (A2) and (A3) together with the assumptions on $b_i(x)$ ensure that there exists $c\geq 0$ such that $\textrm{drift}[U]_t\leq ct$, which implies finiteness of the finite-variation part of $U$ whenever $t$ is bounded. Applying McKean argument we obtain the result. 
\end{proof}

\subsection{Instant diffraction}
Now we consider the case when $V_N(0)=0$. If the process starts from a collision point, then we must first study the question of the instant diffraction (i.e. becoming different) of the particles. We begin with showing that under certain conditions imposed on the coefficients of the equation such phenomenon takes place, i.e. the stopping times  
\formula{
\tau_n= \inf\{t>0: V_n(t)>0\}, \ n=1,\dots, N,
}
are $0$ with probability one.

%

\begin{proposition}
\label{prop:collisiontime:glued}
   Let $(V_1,\ldots,V_N)$ be  semimartingales described by (\ref{eq:Vn:martingale}) and (\ref{eq:Vn:drift}) and such that $V_N(0)= 0$, where $\lambda_i=f_i(Y)$ and
$Y=(y_1,\dots,y_p)$ is a solution of the SDEs system (\ref{eq:en:SDEs}).
	If (A4) holds then $\tau_N=0$ almost surely.

\end{proposition}
\begin{proof}
First we assume that if $\lambda_i(0)=\lambda_j(0)=x\in \R$ for some $i\neq j$, then $\sigma_i^2(x)+
\sigma_j^2(x)+H_{ij}(x,x)>0$, i.e. $\Lambda(0)$ belongs to the set
\formula{
   E = \{(x_1,\ldots,x_p)\in\R^p: (x_i\neq x_j)\vee (\sigma_i^2(x_i)+\sigma_j^2(x_i) +H_{ij}(x_i,x_j)>0)\/,\quad \textrm{ for every }i\neq j\}\/.
}
Since the functions $\sigma_i^2(x)$ and $H_{ij}(x,x)$ are continuous, the set $E$ is open. The continuity of the paths implies that there exists a positive stopping time $\tau$ such that on the interval $[0,\tau)$ the system stays in $E$, i.e. if $ \lambda_i(0)\neq \lambda_j(0)$, then
\formula[eq:cond1]{
    \lambda_i(t)\neq \lambda_j(t)\/,\quad t<\tau
}
and if $\lambda_i(0)=\lambda_j(0)$ for some $i\neq j$, then
\formula[eq:cond2]{
   \sigma_i^2(\lambda_i(t))+\sigma_j^2(\lambda_i(t))    +H_{ij}(\lambda_i(t),\lambda_i(t))>0\/,\quad t<\tau\/.
}
Now inductively we show that if $\tau_n>0$ with positive probability, then the probability that $\tau_{n-1}>0$ is also positive. Note that if $V_n(t)=0$ on $[0,\tau_n)$, then its finite-variation part vanishes on $[0,\tau_n\wedge\tau)$.
Since then $(\lambda_i-\lambda_j)e_{n-1}^{\overline{a}_{ij}}(A)=0$ and $(\lambda_i-\lambda_j)(\lambda_i-\lambda_k)e_{n-2}^{\overline{a}_{ij},\overline{a}_{ik}}(A)=0$ using (\ref{eq:Vn:drift}) we obtain that
\formula[eq:proof:estimate1]{
   \sum_{i=1}^p\sigma_i^2(\lambda_i)\sum_{j\neq i}e_{n-1}^{\overline{a}_{ij}}(A)+4\sum_{i<j}H_{ij}(\lambda_i,\lambda_j)e_{n-1}^{\overline{a}_{ij}}(A)= D_n =0
}
Now let us fix $i\neq j$. If $\lambda_{i}(t)= \lambda_j(t)$ at $t\in[0,\tau_n\wedge\tau)$ then by (\ref{eq:cond1}) we have $\lambda_{i}(0)= \lambda_j(0)$ and consequently by (\ref{eq:cond2}) one of the functions $\sigma_i^2(x)$, $H_{ij}(x,x)$ is positive at $\lambda_i(t)$. Then, the above-given equality implies that  $e_{n-1}^{\overline{a}_{ij}}(A)=0$ at $t$. If $\lambda_{i}(t)\neq \lambda_j(t)$, since $V_n(t)=0$ and
\formula{
   V_n(t)=a_{ij}e_{n-1}^{\overline{a}_{ij}}(A)+e_{n}^{\overline{a}_{ij}}(A)\/,
}
we also get $e_{n-1}^{\overline{a}_{ij}}(A)=0$. It means that for every $t\in [0,\tau_n\wedge\tau)$ and every $i$ and $j$ we have $e_{n-1}^{\overline{a}_{ij}}(A)=0$ which implies that $V_{n-1}(t)=0$ on this interval. 

Finally, if we assume that $\tau_N>0$ with positive probability, then the first part of the proof implies that $\tau_1>0$ with positive probability, i.e. if a pair of particles remains glued for some positive time, then all of the particles are glued for some time. But for $n=1$ the formula given in (\ref{eq:proof:estimate1}) reads as
\formula{
    (p-1)\sum_{i=1}^p\sigma_i^2(\lambda_i)+4\sum_{i<j}H_{ij}(\lambda_i,\lambda_j)= D_1 =0
}
for every $t\in[0,\tau_1\wedge \tau)$ and it is a contradiction with (\ref{eq:cond2}). Note that if there exists $i\neq j$ such that $\lambda_i(0)\neq \lambda_j(0)$ (i.e. $V_n(0)>0$ for some $n\geq 1$) then this finite induction can be stopped at level $n$, because we obtain then a contradiction with the continuity of the paths.

Now we consider the remaining case $\Lambda(0)\notin E$. It means that the process $\Lambda$ starts from a multiple degenerate point. By condition  (A4) and Proposition \ref{prop:Degenerated} we obtain that the system immediately visits the set $E$. Thus, the standard argument based on the Markov property and the continuity of the paths together with the above-given proof for the case when $\Lambda(0)\in E$ give
\formula{
\pr_{\Lambda(0)}(\tau_N>t) &\leq \sum_{q\in[0,t)\cap \Q} \pr_{\Lambda(0)}(\tau_N>t,\Lambda(q)\in E)\\
& = \sum_{q\in[0,t)\cap \Q} \ex_{\Lambda(0)}[\pr_{\Lambda(q)}(\tau_N>t-q), \Lambda(q)\in E] = 0\/.
}
Consequently $\pr_{\Lambda(0)}(\tau_N=0)=1$ almost surely even if $\Lambda(0)\notin E$. This ends the proof.
\end{proof}
\subsection{No collision after instant diffraction}
Now we can state the main result of this section.
\begin{theorem}
\label{thm:TN:infty}
    Let $(V_1,\ldots,V_N)$ be  semimartingales described by (\ref{eq:Vn:martingale}) and (\ref{eq:Vn:drift}) and such that $V_n(0)\geq 0$ for $n=1,\ldots,N$, where $\lambda_i=f_i(Y)$ and
$Y=(y_1,\dots,y_p)$ is a solution of the SDEs system (\ref{eq:en:SDEs}). If the assumptions (A2)-(A5) hold then $T=\infty$ almost surely.
\end{theorem} 
\begin{proof}
   If $V_N(0)>0$ then it is just the result given in Proposition \ref{prop:collisiontime:distinct}. If $V_N(0)=0$, then for every $t>0$, by continuity of the paths and Proposition \ref{prop:collisiontime:glued}, we have
	\formula{
	  \pr_0\left( \bigcap_{q\in[0,t)\cap\Q}\{V_N(q)=0\}\right)=0\/.
	}
	Consequently, by the
	Markov property
	\formula{
	  \pr_0(V_N(t)=0) &\leq \sum_{q\in[0,t)\cap\Q}\pr_0(V_N(t)=0,V_N(q)>0)\\
		&= \sum_{q\in[0,t)\cap\Q}\ex_0(\pr_{V_N(q)}[V_N(t-q)=0], V_N(q)>0)=0\/,
	}
	where the last equality follows from Proposition \ref{prop:collisiontime:distinct}. If we now define for every $s\geq 0$ 
	\formula{
	 T^s = \inf\{t>s:V_N(s)=0\}\/,
	}
	then obviously $T^0=T$ is the first collision time and
	\formula{
	   \{{T}=\infty\}=\bigcap_{n=1}^\infty\{{T}^{2/n}=\infty\}\/,\quad \pr_0({T}=\infty) = \lim_{n\to\infty}\pr_0({T}^{2/n}=\infty)\/,
	}
	but once again by the Markov property, Proposition \ref{prop:collisiontime:distinct} and the fact that $V_N(t)>0$ a.s. for every $t>0$, we have
	\formula{
	   \pr_0({T}^{2/n}=\infty) &= \pr_0(V_N(1/n)>0,{T}^{2/n}=\infty) \\
		&= 
		\ex_0(\pr_{V_N(1/n)}[{T}^{1/n}=\infty],V_N(1/n)>0) = \pr_0(V_N(1/n)>0)=1\/.
	}
	The proof is complete.
\end{proof}

\begin{corollary}\label{cor:no_coll_any_soln} 
 Let $(X,B)_{t\geq 0}$   be any  solution of the system (\ref{eq:eigenvalues:SDE:general}) with a starting point $X(0)\in  \overline{C^+}$. Suppose that the conditions (A2)-(A5) hold.
 Then the particles $x_1(t),\ldots,x_p(t))$ never collide for $t>0$.
\end{corollary}

\begin{proof} 
 We consider symmetric polynomials $V_n(t)$ of squares of differences between particles $x_1(t),\dots, x_p(t)$. 
All we proved on non-collisions of the process $\Lambda$ in Theorem \ref{thm:TN:infty} and the preceding propositions
was based on the It\^o calculus applied to the system  (\ref{eq:eigenvalues:SDE:general}), so it applies to any
solution
 $X$ of (\ref{eq:eigenvalues:SDE:general}). 
\end{proof}

\section{Existence and uniqueness of a non-colliding solution}\label{WeakExist}
\begin{theorem}\label{th:existence}
  Assume that the conditions (A2)-(A5) hold. Then there exists a continuous solution of (\ref{eq:eigenvalues:SDE:general}) starting from ${\bf x}\in \overline{C^+}$ such that its first collision time is greater than the explosion time of the solution. If additionally (C2) holds, then the explosion time and the first collision time are infinite almost surely. 
\end{theorem}
\begin{proof}
For every $i=1,\ldots,p$ we put
	\formula{
	   \lambda_i(t) := f(y_1(t),\ldots,y_p(t))\/,\quad t\geq 0\/,
	}
	until the first explosion time. Then obviously $\Lambda=(\lambda_1,\ldots,\lambda_p)$ is continuous. Moreover, by Theorem \ref{thm:TN:infty}, we have
	\formula{
	   \lambda_i(t)\neq \lambda_j(t)\/,\quad t>0\/, i\neq j\/.
	}
	Thus, for every $t>s>0$, using the smoothness of $f$ and It\^o formula we have
	\formula[st]{
	   \lambda_i(t)-\lambda_i(s) = \int_s^t \sigma_i(\lambda_i(u))dB_i(u)+\int_s^t \left(b_i(\lambda_i(u))+\sum_{j\neq i}\frac{H_{ij}(\lambda_i(u),\lambda_j(u))}{\lambda_i(u)-\lambda_j(u)}\right)du\/,
	}
	where $B_i$ are one dimensional independent Brownian motions. Here we have used the bijectivity of It\^o formula, i.e. if we apply the It\^o formula for the smooth and invertible function $h$ and semimartingale $X$ and then for $h^{-1}$ and $h(X)$ we arrive at the original semimartingale representation for $X$.  
	
By continuity, $\lambda_i(s)$ tends to $\lambda_i(0)$ whenever $s$ goes to $0$. Moreover,
	$
	\int_0^s \sigma_i(\lambda_i(u))dB_i(u)$ is a continuous martingale starting from 0 and it converges to 0 when $s\rightarrow 0$ almost surely.
	Thus the drift integral in (\ref{st}) converges almost surely when  $s\rightarrow 0$. 
	It means that for $t$ smaller than the explosion time we have
	\formula{
	   \lambda_i(t)-\lambda_i(0)=\int_0^t \sigma_i(\lambda_i(u))dB_i(u)+\int_0^t \left(b_i(\lambda_i(u))+\sum_{j\neq i}\frac{H_{ij}(\lambda_i(u),\lambda_j(u))}{\lambda_i(u)-\lambda_j(u)}\right)du\/,
	}
	where the last integral is understood as an improper integral whenever ${\bf x}\notin {C^+}$. Obviously, if (C2) holds, then the explosion time and consequently the first collision time are infinite. This ends the proof.
\end{proof}
\begin{remark}\label{cite_sasiad}
  Note that the assumptions (A2)-(A5) were used only to ensure that $V_N(t)>0$ for every $t>0$. So even if (A2)-(A5)  do not hold, but we can show that $V_N(t)>0$ for every $t>0$, then we can construct a solution of (\ref{eq:eigenvalues:SDE:general}) in the way described above.
\end{remark}

In the next Theorem we use the conditions (C1) and (A1) to show the pathwise uniqueness of the solutions of (\ref{eq:eigenvalues:SDE:general}).


\begin{theorem}\label{th:Pathwise_uniqueness}
Assume that the assumptions (A1)-(A5) and (C1) hold. Then the pathwise uniqueness for solutions of the system (\ref{eq:eigenvalues:SDE:general}) with $X(0)\in \overline{C_+}$  holds.
\end{theorem}
\begin{proof}
Let $(X,B)$ and $(\tilde{X},B)$ be two solutions of (\ref{eq:eigenvalues:SDE:general}) having common starting point $X(0)=\tilde{X}(0)\in  \overline{C^+}$, where $B$ is a Brownian motion in $\R^p$. Corollary \ref{cor:no_coll_any_soln} implies that the particles $X=(x_1,\ldots,x_p)$ do not collide after the start and the same is true for  $\tilde{X}=(\tilde{x}_1,\ldots,\tilde{x}_p)$.

 The condition (C1) together with Lemma $3.3$ from \cite{bib:ry99}, p. $389$, implies that the local time of $Z_i=x_i-\tilde{x}_i$ at 0 is zero. Note that we can apply the Tanaka formula to the process $Z_i$
(if ever the drift integral in the SDE for $Z_i$ is an improper integral in 0,  we write the Tanaka formula for $|Z_i|$ on $[s,t]$ and consider $s\rightarrow 0$. The local time
at 0 of the process $(Z_i(u))_{s\leq u\leq t}$ converges a.s. to the local time
at 0 of the process $(Z_i(u))_{0\leq u\leq t}$). Thus, 
\formula{
\sum_{i=1}^p \ex|x_i(t)-\tilde{x_i}(t)| &= \ex \int_0^t \sum_{i=1}^p\textrm{sgn}(x_i-\tilde{x}_i)\sum_{j\neq i}\left(\frac{H_{ij}(x_i,{x_j})}{x_i-x_j}-\frac{H_{ij}(\tilde{x_i},\tilde{x_j})}{\tilde{x_i}-\tilde{x_j}}\right)du\\
&+\ex\int_0^t \sum_{i=1}^p\textrm{sgn}(x_i-\tilde{x}_i)(b_i(x_i)-b_i(\tilde{x_i}))du\/.
}
The Lipschitz condition imposed on $b_i(x)$ implies that there exists $c\geq 0$ such that
\formula{
  \ex\int_0^t \sum_{i=1}^p\textrm{sgn}\left(x_i-\tilde{x}_i\right)(b_i(x_i)-b_i(\tilde{x_i}))du\leq c\ex \int_0^t\sum_{i=1}^p|x_i-\tilde{x}_i|du\/.
}
It is also true if $b_i(x)$ is non-increasing, since in this 
case
$\textrm{sgn}\left(x_i-\tilde{x}_i\right)(b_i(x_i)-b_i(\tilde{x_i}))\le 0.
$ Moreover, the assumptions (A1) and  (\ref{eq:H:symmetry}) on the functions $H_{ij}$ ensure that the first term is non-positive. Indeed, we can write it in the following form
\formula[Fij]{
 \ex \int_0^t \sum_{i<j}^p\left[\textrm{sgn}\left(x_i-\tilde{x}_i\right)F_{ij}+\textrm{sgn}\left(x_j-\tilde{x}_j\right)F_{ji}\right]du\/,
}
where 
\formula{
F_{ij} = \frac{H_{ij}(x_i,{x_j})}{x_i-x_j}-\frac{H_{ij}(\tilde{x_i},\tilde{x_j})}{\tilde{x_i}-\tilde{x_j}}\/.
}
Note that by the symmetry property   (\ref{eq:H:symmetry}) we have $F_{ij}+F_{ji}=0$.
If $\textrm{sgn}(x_i-\tilde{x}_i)=\textrm{sgn}(x_j-\tilde{x}_j)$, it follows that the term indexed by $i,j$ in (\ref{Fij}) vanishes.  
If the signs of the differences between particles with and without tilde are different, then the term indexed by $i,j$ in (\ref{Fij}) equals $2\textrm{sgn}\left(x_i-\tilde{x}_i\right)F_{ij}$ and (A1) implies that 
\formula{
   \textrm{sgn}(x_i-\tilde{x}_i)F_{ij}\leq 0\/.
}
Consequently, we have obtained that
\formula{
  \sum_{i=1}^p \ex|x_i(t)-\tilde{x_i}(t)|\leq c \ex \int_0^t\sum_{i=1}^p|x_i(u)-\tilde{x_i}(u)|du
}
and the Gronwall Lemma ends the proof.
\end{proof}

{\it Proof of Theorem \ref{thm:main}}.
Theorems \ref{th:existence} and  \ref{th:Pathwise_uniqueness} imply (see Theorem (1.7), p.368 in \cite{bib:ry99}) that if all the conditions
(C1), (C2) and (A1)-(A5) hold, then  the system  (\ref{eq:eigenvalues:SDE:general})  with $X(0)\in  \overline{C_+}$
has a strong pathwise unique solution, with no collisions after having started from $t=0$ and with no explosions for $t>0$.

\section{Examples and applications}\label{sec:EandA} 

 \subsection{Case $\sigma,b,H$}\label{simplerCase} 
The hypotheses of our results simplify when instead of families of functions $(\sigma_i), (b_i), (H_{ij}),\ i,j=1,\dots,p$
we consider the same continuous functions $\sigma(x), b(x)$  and  $H(x,y)=H(y,x)$. Then the equations (\ref{eq:eigenvalues:SDE:general}) simplify to
\formula[eq:eigenvalues:SDE:general:simplified]{
&dx_i = \sigma(x_i)dB_i+\left(b(x_i)+\sum_{j\neq i}\frac{H(x_i,x_j)}{x_i-x_j}\right)dt\/,\quad i=1,\ldots,p\/,\\
& x_1(t)\leq \ldots\leq x_p(t),\ \ \  t\geq 0 \nonumber. 
}
 Note that the systems (\ref{eq:eigenvalues:SDE:general:simplified}) contain as a special case  the systems  (\ref{eq:eigenvalues:SDE})
 related to eigenvalues of matrix stochastic processes. The conditions (A1)-(A5) may be  simplified to

		\begin{itemize}
		  \item[(A1')]  For all  $w<x<y<z$
			$$
			   \frac{H(w,z)}{z-w}\leq \frac{H(x,y)}{y-x} .
			$$
			\item[(A2')] For all  $x,y$  
			$$
			   \sigma^2(x)+\sigma^2(y)\leq 4 H(x,y)
			$$
			 
			\item[(A3')] For all  $x<y<z$ 
			$$
			  H(y,z)(z-y)\leq H(x,y)(y-x)+H(x,z)(z-x)
			$$
	
		  \item[(A4')] For all  $x$  
			$$
			  \sigma^2(x)+H(x,x)>0
			$$
			 or, otherwise,  for every $y_1,\ldots,y_{p-2}\in \R $
			$$
			   b(x)+\sum_{j}\frac{H(x,y_j)}{x-y_j}\ind_{\R\setminus \{x\}}(y_j) \neq 0.
			$$
		\end{itemize}
 \begin{corollary}
 Suppose that $\sigma$ is at least 1/2-H\"older and $b$ is Lipschitz (see condition (C1)) and that the condition (C2) of non-explosion holds.
 If the conditions (A1')-(A4') are verified,
then the system  (\ref{eq:eigenvalues:SDE:general:simplified}) with $X(0)\in  \overline{C_+}$ 
has a strong pathwise unique solution, with no collisions and  no explosions for $t>0$.
 \end{corollary}

  \subsection{Interacting Brownian particles}

In this subsection  we consider the following interacting Brownian particle systems including  and essentially bigger than Dyson   Brownian particle  systems
and the systems considered by C\'epa, L\'epingle (see \cite{bib:cepa}):

\formula[eq:DysonCepaL]{
&dx_i = \sigma_i(x_i)dB_i+\left(b_i(x_i)+\gamma \sum_{j\neq i}\frac{1}{x_i-x_j}\right)dt\/,\quad i=1,\ldots,p\/,\\
& x_1(t)\leq \ldots\leq x_p(t),\ \ \  t\geq 0 \nonumber. 
}

 \begin{corollary}
  Let the functions $\sigma_i$ be at least 1/2-H\"older and $b_i$ be Lipschitz (i.e. they verify  condition (C1)), with $b_i(x)\le b_j(x)$ if $i<j$, and let $\gamma>0$.
  We suppose that the conditions
  $$
   b_i(x)x\leq c(1+|x|^2)\/, \ \ 
   \sigma_i^2(x) \le 2\gamma,\ \ i=1,\dots,p
  $$
  hold for all $x\in\R$. Then the system  (\ref{eq:DysonCepaL}) with $x(0)\in  \overline{C_+}$ 
has a strong pathwise unique solution, with no collisions and  no explosions for $t>0$.
 \end{corollary}
 \begin{proof}
  We apply the Theorem 1  with constant positive $H=\gamma$. The conditions (A1), (A3) and (A4) are  satisfied (observe that the sets  $G_{kl}$ from (A4) are empty.)
  The rest of the assumptions of Theorem 1 hold thanks to the assumptions of the Corollary.
 \end{proof}

 \subsection{ Brownian particles with nearest neighbor repulsion}\label{Se:neighb}
 
 Consider  following systems of Brownian particles where only neighbor particles 
 are interacting and the repelling force is proportional to the inverse of the distance between particles:
\formula[eq:Neighbour]{
&dx_1 = \sigma_1(x_1)dB_1+  \frac{\gamma}{x_1-x_{2}} dt, \nonumber  \\
&dx_i = \sigma_i(x_i)dB_i+ \gamma  \left(\frac{1}{x_i-x_{i-1}}+\frac{1}{x_i-x_{i+1}} \right)dt\/,\quad i=2,\ldots,p-1\/,\\
&dx_p = \sigma_p(x_p)dB_p+  \frac{\gamma}{x_p-x_{p-1}}dt,  \nonumber \\
& x_1(t)\leq \ldots\leq x_p(t),\ \ \  t\geq 0 \nonumber. 
}
Here, the functions $H_{ij}= \gamma$ when $|i-j|=1$ and they are zero otherwise.

Note that in this case the condition (A3) does not hold. Since (A3) was only used to show that the particles starting from non-collision points do not collide, it is enough to prove this fact directly. Because the proof in the general case is very technical, we only deal with the case $p=3$ in the next corollary. However, the proof technique presented below can also be applied to the general case $p\geq 4$. 

 \begin{corollary}
  Let the functions $\sigma_i$ be at least 1/2-H\"older and such that $|\sigma_i(x)|\leq 1$. If $p=3$ and $\gamma\geq 3/4$
then the system  (\ref{eq:Neighbour}) with $x(0)\in  \overline{C_+}$ 
has a strong pathwise unique solution, with no collisions and  no explosions for $t>0$.
 \end{corollary}
\begin{proof}
  We will show that the drift part of the semimartingale $U_t$ defined in the proof of Proposition \ref{prop:collisiontime:distinct} is non-positive. Indeed, using the bounds $|\sigma_i(x)|\leq 1$ we get
	\formula{
	  \textrm{drift}[U]_t \leq (1-2\gamma)\left(\frac{1}{(x_2-x_1)^2}+\frac{1}{(x_3-x_2)^2}\right) + \frac{1}{(x_3-x_1)^2}+\frac{\gamma}{(x_2-x_1)(x_3-x_2)}
	}
	Since
	$\frac{1}{a^2}+\frac{1}{b^2}\geq \frac{2}{ab}$ and $\frac{1}{(a+b)^2}\leq \frac{1}{4ab}$ whenever $a,b>0$ and obviously $x_3-x_1=(x_2-x_1)+(x_3-x_2)$, we arrive for $\gamma\geq 1/2$ at
	\formula{
	   \textrm{drift}[U]_t&\leq
	   \frac{2-3\gamma}{(x_2-x_1)(x_3-x_2)}+\frac{1}{4(x_2-x_1)(x_3-x_2)}= \frac{9-12\gamma}{4(x_2-x_1)(x_3-x_2)}
	} 
	which is non-positive if $\gamma\geq 3/4$. It means that even though the condition (A3) does not hold, the assertion of Proposition \ref{prop:collisiontime:distinct} is true. This ends the proof.
\end{proof}
\remark We conjecture that the condition 
	\formula{
	   \gamma\geq \frac{p}{2}\left(\sum_{i=1}^{p-1}\frac{1}{i^2}\right)\left(\sum_{i=1}^{p-1}\frac{1}{i}\right)^{-1}-\frac12\/,
	}
	ensures non-positivity of the drift part of $U_t$ in the general case $p\geq 3$.
	
 The terminology ``Brownian particles with nearest neighbor repulsion'' was used in \cite{bib:RostV} and \cite[Section 5.1]{bib:lep}
 to the systems of the form
 \formula[eq:NeighbourOld]{
 &dX_1 =  dB_1+  \phi'(X_1-X_{2}) dt, \nonumber  \\
&dX_i = dB_i+  \phi'( X_i-X_{i-1})-\phi'(X_i-X_{i+1} )dt\/,\quad i=2,\ldots,p-1\/,\\
&dX_p = dB_p+\phi'( X_p-X_{p-1})dt,  \nonumber \\
& X_1(t)\leq \ldots\leq X_p(t),\ \ \  t\geq 0\/, \nonumber 
 }
 where $\phi$ is a positive convex function on $(0,\infty)$ satisfying $\phi(0)=\infty$, $\phi(\infty)= 0$ and the non-collision condition $\int_{0+} \exp(2\phi)=\infty$ (this condition is stronger in
  \cite{bib:RostV}).
 Observe that the system (\ref{eq:Neighbour}) is not contained in systems  (\ref{eq:NeighbourOld}).

 \subsection{ Non-colliding Squared Bessel particles and related processes}
In this section we consider the processes satisfying the following system of SDEs

\begin{equation}\label{eq:wishartBeta}
dx_i= \sigma_i(x_i) dB_i +\beta\left( \alpha    +\sum_{k\not=i} \frac{ x_i+x_k
} {x_i-x_k}\right) dt,\ \ \beta>0.    
 \end{equation}
When $\sigma_i(x)=2 \sqrt {x} $, $i=1,\dots,p$, these processes are called $\beta$-Wishart processes
and contain for
$\beta=2$ the non-colliding Squared Bessel particle systems studied in \cite{bib:katori2011}.
For  applications of these classes of particle systems, see \cite{bib:gm13} and \cite{bib:katori2011}. The 
$\beta$-Wishart processes were studied in \cite{bib:demni}.

 \begin{corollary}\label{Wishart}
 Let $\alpha\geq p-1$ and $\beta\geq 1$. Suppose that  the functions $\sigma_i$  are defined on $\R$ and  verify the condition (C1) and the estimate
 \formula[ineq]{ \sigma_i(x)^2\le 4\beta |x|,\ \ x\in \R.} 
Then
	the system (\ref{eq:wishartBeta}) with an initial condition $0\leq x_1(0)\leq x_2(0)\leq\ldots\leq x_p(0)$ has a unique strong solution for $t\in[0,\infty)$.  Moreover, the process $x_1$ verifies  $x_1(t)\geq 0$ and
 there are  no collisions between the processes $x_i(t)$ for $t>0$.
 \end{corollary}
\begin{proof}
We apply Theorem 1 with  
 the functions $H_{ij}(x,y)= |x|+|y|$  not depending  on $i,j$. To see that (A1) holds note that the trapezium with vertices
$(x,0)$, $(x,|x|)$, $(y,|y|)$ and $(y,0)$ is included in the  trapezium with vertices $(w,0)$, $(w,|w|)$, $(z,|z|)$ and $(z,0)$, whenever $w<x<y<z$. Condition (A2) follows from  inequality  (\ref{ineq}). We prove 
 condition (A3) in a similar way as  Corollary 1 of \cite{bib:gm13}.
As inequality  (\ref{ineq}) implies that $\sigma_i(0)=0$, the sets $G_{kl}$ are equal to $\{0\}$ and condition (A4) holds
since $\alpha\notin\{0,1,2,\dots, p-2\}$.
 Similarly as in  Theorem 7 and Proposition 1 of \cite{bib:gm13}, the condition
 $\alpha\geq p-1$ guarantees that  $x_1(t)\geq 0$.
\end{proof}
Note that the Corollary \ref{Wishart} strengthens Corollary 6 of \cite{bib:gm13}.  When $\sigma_i(x)=2 \sqrt {x} $ and $\alpha>p$,
  Corollary \ref{Wishart} was proved in \cite{bib:lep} by the methods of MSDEs, see also \cite{bib:demni}.

The proof of the last  corollary   applies to more general SDEs systems of the form
\begin{equation}\label{eq:wishartBetaABSVAL}
dx_i=\sigma_i(x_i) dB_i +\beta\left( \alpha    +\sum_{k\not=i} \frac{ |x_i|+|x_k|
} {x_i-x_k}\right) dt,\ \ \beta>0   
 \end{equation}
and we obtain  the following corollary strengthening Corollary 4 of \cite{bib:gm13}. 
 \begin{corollary}\label{WishartGEN}
 Let $\alpha\in\R\setminus\{0,1,2,\ldots, p-2\}$ and $\beta\geq 1$. Suppose that  the functions $\sigma_i$ verify  conditions (C1) and 
(\ref{ineq}).
Then
	the system (\ref{eq:wishartBetaABSVAL}) with an initial condition $ x_1(0)\leq x_2(0)\leq\ldots\leq x_p(0)$ has a unique strong solution for $t\in[0,\infty)$. 
 There are  no collisions between the processes $x_i(t)$ for $t>0$.
 \end{corollary}

The generalized Squared Bessel particle systems of the form (\ref{eq:wishartBetaABSVAL})  for any $\alpha\in\R$ will be studied
in a forthcoming paper \cite{bib:gmTHIRD}.  On the other hand, defining $y_i=\sqrt{x_i}$ where the processes $x_i$ are solutions of the system
(\ref{eq:wishartBeta}) with $\sigma_i(x)=2\sqrt{x}$,  we obtain   Bessel particle systems and the results of Corollary
  \ref{Wishart} can be transfered to those systems, cf. \cite{bib:lep}.

 \subsection{ Non-colliding Jacobi particles}
The methods of this paper can also be applied to non-colliding Jacobi particle systems on the segment $[0,1]$,  defined by
 \begin{equation}\label{eq:lambdaJacobiBeta}
  dx_i=2 \sqrt {x_i(1-x_i)} dB_i +\beta\left( q-(q+r)x_i   +\sum_{k\not=i} \frac{ x_i(1-x_k)+x_k(1-x_i)
} {x_i-x_k}\right) dt.    
 \end{equation}
 Observe that  the sets $G_{kl}=\{0,1\}$ in this case. Corollary 8 of \cite{bib:gm13}
generalizes to the case $0\leq  x_1(0)\leq x_2(0)\leq\ldots\leq x_p(0)\leq 1$.
\begin{corollary}
  The SDE system (\ref{eq:lambdaJacobiBeta}) with $0\leq  x_1(0)\leq x_2(0)\leq\ldots\leq x_p(0)\leq 1 $ has a unique strong solution for $t\in[0,\infty)$, 
  for any  $\beta\ge 1$ and  $q\wedge r\ge p-1$.
 \end{corollary}

 \subsection{Hyperbolic  particle systems}\label{trig}
   The hyperbolic particle systems have the form
 \begin{eqnarray}\label{eq:Hyperbolic}
&&dx_i= \sigma_i(x_i) dB_i +  
\left(b_i(x_i)+\gamma \sum_{j\neq i} \coth ({x_i-x_j})\right)dt\/,\quad i=1,\ldots,p\/,\\
&& x_1(t)\leq \ldots\leq x_p(t),\ \ \  t\geq 0 \nonumber. 
 \end{eqnarray}
 In the special case $\sigma_i=1, b_i=0$ they arise as radial Heckman-Opdam processes and were studied in \cite{bib:Schapira}.
 \begin{corollary}
  Let the functions $\sigma_i$ be at least 1/2-H\"older and $b_i$ be Lipschitz (i.e. they verify  condition (C1)), with $b_i(x)\le b_j(x)$ if $i<j$, and let $\gamma>0$.
   We suppose that the conditions
  $$
   b_i(x)x\leq c(1+|x|^2)\/, \ \ 
   \sigma_i^2(x) \le 2\gamma,\ \ i=1,\dots,p
  $$
  hold for all $x\in\R$. Then the system  (\ref{eq:Hyperbolic}) with $x(0)\in  \overline{C_+}$ 
has a strong pathwise unique solution, with no collisions and  no explosions for $t>0$.
 \end{corollary}
 \begin{proof}
  Consider the function $h(x)=\gamma x\coth x$. By continuity, $h(0)=\gamma$.
   We can  apply Theorem \ref{thm:main} with $H(x,y)=h(y-x)$. Indeed, 
  Condition (A1) holds because the function $\coth x$ is decreasing on $\R^+$.
 The inequality $\tanh x\leq x$ for $x\ge 0$ implies that $h(x)\geq \gamma$ and the assumption (A2) is satisfied. 
  Condition (A3) is true since $\coth(a+b)=(1+\coth a \coth b)(\coth a+ \coth b)> \coth a+ \coth b$ for $a,b>0$. The sets $G_{kl}$ are empty because
 $H(x,x)=h(0)=1$. The Corollary follows.
 \end{proof}
 
Note that the same proof works in much greater generality and gives the following Corollary.
\begin{corollary}
 Consider a system of SDEs
 \begin{eqnarray} \label{eq:HyperbolicBis}
&&dx_i= \sigma_i(x_i) dB_i +  
\left(b_i(x_i)+ \sum_{j\neq i} \psi ({x_i-x_j})\right)dt\/,\quad i=1,\ldots,p\/,\\
&& x_1(t)\leq \ldots\leq x_p(t),\ \ \  t\geq 0 \nonumber, 
 \end{eqnarray}
 where   $\psi$ is a continuous odd function  which is non-negative and decreasing on $\R^+$, with $\psi(0)=\infty$,  $\psi(x+y)\ge \psi(x)+\psi(y)$
 and $x\psi(x)\geq \gamma>0$ for $x,y\geq 0$. 
 Let the functions $\sigma_i$ be at least 1/2-H\"older and $b_i$ be Lipschitz (i.e. they verify  condition (C1)), with $b_i(x)\le b_j(x)$ if $i<j$.
   We suppose that the conditions
  $$
   b_i(x)x\leq c(1+|x|^2)\/, \ \ 
   \sigma_i^2(x) \le 2\gamma,\ \ i=1,\dots,p
  $$
  hold for all $x\in\R$. Then the system  (\ref{eq:HyperbolicBis}) with $x(0)\in  \overline{C_+}$ 
has a strong pathwise unique solution, with no collisions and  no explosions for $t>0$.
\end{corollary}

Finally observe that  the results and techniques of this section and Section \ref{Se:neighb} may be applied to systems (\ref{eq:NeighbourOld}),
e.g. for $\phi(x)=\ln|\sinh x|$ and $\phi'(x)= \coth x$. 

{\bf Ackowledgements}. We thank   Makoto Katori and Dominique L\'epingle  for stimulating discussions on   particle systems,
that  inspired this work.



\end{document}